\documentclass[11pt]{article}
\usepackage{amsmath}
\usepackage{latexsym}
\usepackage{epsfig}
\usepackage{amssymb}
\begin{document}
\makeatletter
\def\@begintheorem#1#2{\trivlist \item[\hskip \labelsep{\bf #2\ #1.}] \it}
\def\@opargbegintheorem#1#2#3{\trivlist \item[\hskip \labelsep{\bf #2\ #1}\ {\rm (#3).}]\it}
\makeatother
\newtheorem{thm}{Theorem}[section]
\newtheorem{alg}[thm]{Algorithm}
\newtheorem{conj}[thm]{Conjecture}
\newtheorem{lemma}[thm]{Lemma}
\newtheorem{defn}[thm]{Definition}
\newtheorem{cor}[thm]{Corollary}
\newtheorem{exam}[thm]{Example}
\newtheorem{prop}[thm]{Proposition}
\newenvironment{proof}{{\sc Proof}.}{\rule{3mm}{3mm}}

\title{Sequencing Partial Steiner Triple Systems}
\author{Brian Alspach\\School of Mathematical and Physical Sciences\\
University of Newcastle\\Callaghan, NSW 2308,
Australia\\brian.alspach@newcastle.edu.au\\ \\
Donald L. Kreher\\
Department of Mathematical Sciences\\
Michigan Technological University\\
Houghton, Michigan 49931\\kreher@mtu.edu\\ \\
Adri\'an Pastine\\Instituto de Matem\'atica Applicada San Luis\\Universidad Nacional de San Luis\\
San Luis, Argentina\\adrian.pastine.tag@gmail.com}
\maketitle

\begin{abstract} A partial Steiner triple system of order $n$ is sequenceable if there is a sequence
of length $n$ of its distinct points such that no proper segment of the sequence is a union of
point-disjoint blocks.  We prove that if a partial Steiner triple system has at most three point-disjoint
blocks, then it is sequenceable.
\end{abstract} 

{\small{\bf Keywords}: partial Steiner triple system, sequenceable.

\section{Introduction}

An approach using posets for the problem of determining whether abelian groups are
strongly sequenceable has been introduced in \cite{A9}.  The definition of a strongly sequenceable
abelian group is not relevant for us, but what is relevant is that the approach involves extracting
a poset from a combinatorial problem and working with the poset.  The purpose of this paper is
to apply the approach to partial Steiner triple systems.   Following are several definitions
to establish the landscape.  

Given a sequence $\pi=s_1,s_2,\ldots,s_n$, a {\it segment} is a subsequence of consecutive 
entries of $\pi$.  We do not use the term `interval' because that term has a different meaning
in the context of posets.  If the segment has $t$ terms, it sometimes is called a $t$-{\it segment}.
A segment is {\it proper} if it is neither empty nor $\pi$ itself.

\begin{defn}{\em Let $\mathcal{P}$ be a poset on a ground set $\Omega$, where the elements
of $\mathcal{P}$ are subsets of $\Omega$.  If there is a sequence $$\pi=s_1,s_2,\ldots,s_t$$
of the distinct elements of $\Omega$ such that no proper segment of $\pi$ belongs to
$\mathcal{P}$, then the poset is} sequenceable.
\end{defn}

When we say that a segment belongs or does not belong to $\mathcal{P}$, we of course mean
the subset of elements comprising the segment.  Note that all of $\Omega$ is allowed to be in
the poset.

By a {\it partial Steiner triple system} we mean a collection of edge-disjoint subgraphs isomorphic to
$K_3$ whose union is the edge set of a graph $X$.  For simplicity we refer to the $K_3$ subgraphs
in the collection as {\it blocks}.  If the graph is complete, a partial Steiner triple system is, in fact, a
Steiner triple system.  The {\it order} of a partial Steiner triple system is the order of the graph $X$,
that is, the number of vertices and we use the notation $\mathrm{PST}(n)$ for a partial Steiner triple
system of order $n$.

\begin{defn}{\em Let $\mathcal{T}$ be a partial Steiner triple system of order $n$.  The} associated
poset $\mathcal{P}(\mathcal{T})$ {\em is the poset with ground set $V(X)$ whose elements are
any subsets which can be partitioned into vertex-disjoint blocks.}
\end{defn}

Note that the blocks themselves are minimal elements of the poset $\mathcal{P}(\mathcal{T})$.
Also note that the cardinalities of the elements of $\mathcal{P}(\mathcal{T})$ are multiples of 3. 

\begin{defn}{\em A partial Steiner triple system $\mathcal{T}$ is} sequenceable {\em if the associated
poset $\mathcal{P}(\mathcal{T})$ is sequenceable.}
\end{defn}

Distinct blocks are edge-disjoint by definition so that involving the word disjoint for this situation
is redundant.  Consequently, and this is a point of emphasis, when we say two blocks are disjoint,
we mean they are vertex-disjoint.

A sequence of the distinct
elements of a poset that satisfies the definition of sequenceability is said to be {\it admissible}.  In
accordance with this, a proper segment of a sequence is {\it admissible} if it contains no segments
(including itself) that belong to the poset.  A segment that is not admissible is {\it inadmissible}.
Thus, given a $\mathrm{PST}(n)$, we shall be looking for a sequence of the distinct elements of the 
vertex set $V(X)$ such that all proper segments are admissible, that is, no proper segment of the
sequence has a partition into disjoint blocks.  One common technique is to make small perturbations
of a given sequence that eliminate all inadmissible segments.

We write blocks using square brackets in order to distinguish them from arbitrary subsets of vertices of cardinality
three.  We use $V$ to denote $V(X)$ throughout the rest of the paper unless confusion would result.

\section{Universality}

When one first learns of this problem, the immediate thought is ``Wouldn't it be nice if every partial
Steiner triple system is sequenceable.''  Alas, this is not the case as we soon shall see, but some questions
arise and we present them in this short section.

\begin{thm}\label{13} The cyclic Steiner triple system of order 13 is not sequenceable.
\end{thm}
\begin{proof}  The cyclic Steiner triple system of order 13 admits the permutation $\rho=(0\;\;1\;\;
2\;\;3\;\;4\;\;5\;\;6\;\;7\;\;8\;\;9\;\;10\;\;11\;\;12)$ as an automorphism. The blocks are obtained
by the action of the group $\langle\rho\rangle$ on the two blocks $[0,1,4]$ and $[0,2,7]$.  Then
the four blocks $[0,2,7],[1,3,8],[5,6,9]\mbox{ and }[4,10,12]$ are disjoint and miss
the vertex $11$.  Hence, for any vertex $i$ there is an appropriate power of $\rho$ that maps the
preceding four disjoint blocks to four disjoint blocks missing $i$.  Therefore, every
sequence of the distinct vertices has two inadmissible 12-segments, namely, those obtained by
deleting the first vertex and the last vertex.    \end{proof}

\medskip

The preceding theorem has a straightforward proof and certainly destroys any hopes that all
partial Steiner triple systems are sequenceable.  Several questions immediately arise and here are
a few we consider in the remainder of this paper.

\smallskip

{\bf Question 1}.  Is the cyclic Steiner triple system of order 13 the smallest non-sequenceable
$\mathrm{PST}(n)$?

\smallskip

{\bf Question 2}. Is there a nice construction for producing non-sequenceable partial Steiner
triple systems of other orders?

\smallskip

{\bf Question 3}.  Are there conditions on the collection of blocks so that a given $\mathrm{PST}(n)$
is sequenceable?

\section{Disjoint Blocks}

One feature of the $\mathrm{PST}(13)$ used for the proof of Theorem \ref{13} is that there
are four mutually disjoint blocks.  In this section, we explore Question 3 above from the standpoint
of the number of mutually disjoint blocks.  First, we present a general proof strategy which  
becomes useful as the number of disjoint blocks increases.

The $\mathrm{PST}(n)$ is given in terms of one, two or three disjoint blocks and additional
vertices not belonging to the given blocks.  If there is one block, it is denoted $B_1$.  If
there are two disjoint blocks, they are designated $B_1$ and $B_2$.  Three disjoint blocks
are denoted $B_1,B_2\mbox{ and }B_3$.  The set of additional vertices is denoted $L$ in
all cases.

At the beginning none of the vertices of the $\mathrm{PST}(n)$ are labelled.  An initial
labelling is described under the restrictions that the vertices of $B_1$ are labelled 1, 2 and
3; the vertices of $B_2$ are labelled 4, 5 and 6; the vertices of $B_3$ are labelled 7, 8
and 9; and the vertices of $L$ are labelled with lower case letters $a,b,c$ and so on.
After an initial labelling is in place, vertices may be relabelled so that a given sequence
becomes admissible.  The sequence may be given ahead of time or described at some
point later in the process.

\smallskip

The next easy result looks at the case that the
$\mathrm{PST}(n)$ has no two disjoint blocks.   

\begin{lemma}\label{zero}A $\mathrm{PST}(n)$ in which any two blocks have a vertex in common
is sequenceable.
\end{lemma}
\begin{proof} Note that a sequence is inadmissible if and only if there is an inadmissible 3-segment.
This makes checking the sequences straightforward.

If there are no blocks, then every sequence of the ground set is admissible.  So arbitrarily choose
a block to be $B_1$ and arbitrary label its vertices with 1, 2 and 3.  If there are no additional blocks,
then every sequence of the distinct elements is admissible when $V=\{1,2,3\}$.  If there is at least
one additional element $a$ but no additional blocks, then every sequence beginning $1,2,a,3$ is
admissible. 

If there is at least one additional block, then by relabelling, if necessary, we may assume the block
is $[1,a,b]$.  If there are no additional elements, then $1,2,a,b,3$ is an admissible sequence.
If there is at least one additional element $c$, then every sequence beginning $1,c,2,3,a,b$ is
admissible.    \end{proof}

\bigskip

We now move to two disjoint blocks which means that 6-segments may be inadmissible.  The next
result provides useful information about 6-sets.

\begin{lemma}\label{six} Let $A$ be a 6-subset of vertices in a $\mathrm{PST}(n)$.  If $A$ has a
partition into two blocks $A_1$ and $A_2$, then the following are true:

{\rm (1)} The partition is unique;

{\rm (2)} $A$ contains no blocks other than $A_1$ and $A_2$; and

{\rm (3)} A set obtained by replacing any element of $A$ with another element from $V$ does
not have a partition into two blocks.
\end{lemma}
\begin{proof} Given that $A$ partitions into $A_1$ and $A_2$, any other 3-subset of $A$
must intersect one of $A_1$ or $A_2$ in two elements.  This implies the other 3-subset
cannot be a block, thereby, proving (1) and (2).  Part (3) then follows because replacing
a single element of $A$ by an element not in $A$, leaves either $A_1$ or $A_2$ intact.
The latter would then have to be in any partition but the complement then intersects the
original block whose element is removed in two elements so it cannot be a block.    \end{proof}  

\begin{lemma}\label{two} A $\mathrm{PST}(n)$ with at most two disjoint blocks is
sequenceable.
\end{lemma}
\begin{proof} Note that the only possible inadmissible segments in this case are 3-segments and 6-segments. 
If there are no disjoint blocks, then $\mathrm{PST}(n)$ is sequenceable by Lemma
\ref{zero}.  So let $B_1=[1,2,3]$ and $B_2=[4,5,6]$ be two disjoint blocks that have had
their elements arbitrarily labelled.  If $|V|=6$, there are no
additional blocks and $1,2,4,5,3,6$ is an admissible sequence.

If $|V|=7$, let $a$ be the additional vertex.  Additional blocks must include $a$ and such a block involves
an edge in the bipartite graph $K_{3,3}$ with parts $\{1,2,3\}$ and $\{4,5,6\}$.  Thus the maximum number of
additional blocks is three and corresponds to a perfect matching in $K_{3,3}$.   We may assume the additional
blocks are a subset of $\{[1,4,a],[2,5,a],[3,6,a]\}$ by suitably relabelling the elements of
$\{1,2,3\}$ and $\{4,5,6\}$.  An admissible sequence for all possibilities is $1,2,4,a,5,3,6$.

When $|V|=8$, let $a,b$ be the additional vertices. Because we have freedom of choosing which
block is $B_1$ and of labelling the vertices in the two blocks, we may make the following
assumptions.  If $a,b$ belong to a block, we may assume the block is
$[1,a,b]$.  If $3,a$ belong to a block, we may assume it is $[3,6,a]$.  Hence, the sequence
$1,2,4,3,a,5,6,b$ has no 3-segments that are blocks.  It is easy to verify that all three of the
6-segments are admissible. 

Now let $|V|>8$ and $V'=V\setminus\{1,2,3,4,5,6\}$.  There are no
blocks contained in $V'$ as this would give three mutually disjoint blocks contrary to the hypothesis.
Choose $A\subseteq V'$ so that $|A|=3$.  We are free to label the vertices of the two blocks
$[1,2,3]$ and $[4,5,6]$ so that neither $[5,u,v]$ nor $[6,u,v]$ are blocks for any $u,v\in A$.
We also may label the elements of $A$ with $a,b,c$ so that $[3,5,a]$ and $[3,a,b]$ are not blocks.

Let $\pi$ be any sequence beginning $1,2,4,3,5,a,6,b,c$.  It is easy to verify that all 3-segments
are admissible.  No 6-segment disjoint from $1,2,4,3$ is inadmissible as this would give three
mutually disjoint blocks.  The three 6-segments $1,2,4,3,5,a$; $2,4,3,5,a,6$ and $4,3,5,a,6,b$
are easily seen to be admissible.  The only remaining possible inadmissible 6-segment is
$3,5,a,6,b,c$ but it cannot be inadmissible because $[5,u,v]$ and $[6,u,v]$ are not blocks for
any $u,v\in A$.   This completes the proof.       \end{proof}

\medskip

The preceding two results arise from somewhat restrictive conditions.  A natural question is whether the
order of the $\mathrm{PST}(n)$ is bounded.  It is easy to see that the order is not bounded.  The
friendship graph is obtained by taking $m$ 3-cycles and amalgamating them at a common vertex.  the
resulting graph is a $\mathrm{PST}(2m+1)$ and every two blocks intersect in a single vertex.

The friendship graph provides an obvious recipe for obtaining partial Steiner triple systems of
arbitrarily large order such that there are $k$ disjoint blocks but not $k+1$.  Start with $k$
vertex-disjoint friendship graphs $G_1,G_2,\ldots,G_k$ so that there at least two blocks in each
$G_i$.  Then amalgmate $G_i$ and $G_{i+1}$ at a vertex of valency 2 for $i=1,2,\ldots,k-1$
making sure the block you use from $G_{i+1}$ for the amalgamation of $G_i$ and $G_{i+1}$
is different from the block of $G_{i+1}$ you use for $G_{i+1}$ and $G_{i+2}$.  It is easy to see
that there are $k$ disjoint blocks, and there cannot be $k+1$ disjoint blocks by the pigeon-hole
principle.

\section{Three Disjoint Blocks}

In the previous sections, we have shown that a partial Steiner triple system with at most two disjoint
blocks is sequenceable, and that there exists a non-sequenceable partial Steiner triple system with
four disjoint blocks.  This makes three disjoint blocks a focal point and, as we are about to see, the
resolution for the situation is lengthy.  We break up the argument into several lemmas depending
on the order of the $\mathrm{PST}(n)$.   

Reiterating the notation of the proof strategy described earlier, there are three disjoint blocks
denoted $B_1,B_2,B_3$ and the set $L$ of remaining vertices. 
 
\begin{lemma}\label{nine} A partial Steiner triple system of order 9 having three
disjoint blocks is sequenceable.
\end{lemma} 
\begin{proof} The three disjoint blocks are arbitrarily labelled $B_1,B_2,B_3$.  Arbitrarily label the
elements of $B_i$ with $3i-2,3i-1,3i$ for $i=1,2,3$.  Consider the sequence $1,2,4,3,5,7,6,8,9$.
It is easy to see that the only possible inadmissible 3-segment is 3,5,7. ( Note that if $[3,5,7]$ is
a block, that makes all the 6-segments inadmissible.)  

If 3,5,7 is admissible, then every 6-segment is admissible by Lemma \ref{six}.  This implies that
the sequence is admissible.  If $[3,5,7]$ is a block, then interchange the labels 2 and 3.  The
sequence $1,2,4,3,5,7,6,8,9$ is now admissible.    \end{proof}

\medskip

Some additional information is required when extra vertices are added.    
Let $\mathcal{B}$ be the set of blocks of $\mathcal{T}$, where the vertex set $V$ has
cardinality $n$.  The maximum number of blocks is discussed in~\cite{BB}
and is given by

\[|\mathcal{B}|=\left\{\begin{array}{ll}
\left\lfloor \frac{n}{3} \left\lfloor \frac{n-1}{2} \right\rfloor \right\rfloor&
\text{ if $n \equiv 0,1,2,3,4 \pmod{6}$} \\[5pt]
\left\lfloor \frac{n}{3} \left\lfloor \frac{n-1}{2} \right\rfloor \right\rfloor-1&
\text{ if $n \equiv 5 \pmod{6}$} \\
\end{array}\right.\]
This bound is sometimes called the \emph{Johnson-Sch\"onheim bound}.

With respect to a given partial triple system $\mathrm{PST}(n)$, we say that a subset
$M\subset V$ is {\it bad} if $V\setminus M$  can be partitioned into three disjoint blocks;
otherwise, we say the set is a {\it good} set.  We say
that three disjoint blocks $A,B,\mbox{ and }C$ {\it realize} a bad set $M$ if $A\cup B\cup C
=V\setminus M$.  Notice that if $M$ is bad, then $|M|=n-9$.  If $M=\{a\}$ is a singleton set,
we say that $a$ is a {\it bad} point or {\it good} point as appropriate.
 
\begin{lemma}\label{badpoints}  A partial Steiner triple system of order $10$ has at most
four bad points.
\end{lemma}
\begin{proof}  Let $\mathcal{T}$ be a partial triple system of order 10 and assume there are five
bad points.  The Johnson-Sch\"onheim bound implies that $\mathcal{T}$ has at most 13
blocks.  Hence, by the pigeonhole principle, there must be a block $A$ and a pair of bad
points $v,w$, such that $A\cup B \cup C=V\setminus\{v\}$ and $A\cup D\cup E=
V\setminus\{w\}$ for some blocks $B,C,D,E$.

Assume without loss of generality that $v\in D$ and $w\in B$. Then $E\subset B\cup C$,
which implies $E\in \{B,C\}$ by Lemma \ref{six}. But then by counting points we see
that $|B\cap D|=2$, and this is a contradiction.

Therefore, there are  at most four bad points completing the proof.    \end{proof}

\begin{lemma}\label{ten} A $\mathrm{PST}(10)$ with three disjoint blocks is sequenceable.
\end{lemma}
\begin{proof} Let $|V|=10$ and label the single point of $L$ with $a$.  We know there are at 
most four bad points from Lemma \ref{badpoints}.  Because $a$ itself is a bad point, there
are at least six good points among the remaining nine points.  If there is a block with fewer
than two good points, label this block $B_3$; otherwise, arbitrarily label a block as $B_3$.  
Thus, we know the other two blocks each have at least two good points.    

Label an arbitrary point of $B_3$ with 9 and label another block $B_1$ so that $[9,x,a]$ is
not a block for $x\in B_1$.  Arbitrarily label the points of $B_1$ with 1, 2, 3 so that the bad
point (if there are any) has label 3. This means that both 1 and 2 are good points.  In the
block $B_2$, label the points with 4, 5, 6 and in the block $B_3$ label the other two points
with 7 and 8 so that none of $[3,6,8], [3,6,9]\mbox{ or }[3,6,a]$ are blocks.  Finally, from
the good points 1 and 2, relabel the points, if necessary, so that $[2,6,9]$ is not a block.

We claim the sequence $1,4,5,7,6,8,9,3,a,2$ is admissible.  All of the 3-segments are admissible
because the labelling is chosen so that $[9,3,a]$ is not a block.  The two 9-segments are
admissible because 1 and 2 are good points.

The only two possible inadmissible 6-segments are $7,6,8,3,9,a$ and $6,8,3,\\9,a,2$.  The
first is not inadmissible because Lemma \ref{six} implies that $[3,6,a]$ is a block if it is
inadmissible and the latter is not the case.  The only possible partitions of the second
6-segment into two blocks are $[2,8,a]\cup[3,6,9]$ and $[3,8,a]\cup[2,6,9]$, but neither
$[3,6,9]$ nor $[2,6,9]$ are blocks and this completes the verification that $\mathcal{T}$
is sequenceable when there are three disjoint blocks and $|V|=10$.   \end{proof}

\bigskip

Now that we are considering three disjoint blocks, some sets of cardinality nine have
partitions into three blocks.  This type of set is more complicated than the sets of cardinality
six that have partitions into two blocks.  Nevertheless, it would be useful if we could find a
result for nine points that is analogous to the powerful Lemma \ref{six} for six points.  Towards
that end we introduce some notation.

If $A_1$ and $A_2$ are disjoint blocks in a $\mathrm{PST}(n)$, let $K_{A_1,A_2}$ denote the
complete bipartite graph of order 6 with parts $A_1$ and $A_2$.  Let $A$ be a set of cardinality
9 containing both $A_1$ and $A_2$ as subsets, and let $A_3=A\setminus(A_1\cup A_2)$.  

If $A_3$ also is a block, then $A_1,A_2,A_3$ is a partition of $A$ into three blocks.  On the
other hand, if $A_3$ is not a block, then there can be no partition of $A$ having either $A_1$
or $A_2$ as a part because of Lemma \ref{six}.  Thus, given any partition of $A$ into three
blocks in the latter situation, every block in such a partition must intersect each $A_1$ and $A_2$
in a single point.  We may interpret these intersections as a perfect matching in  $K_{A_1,A_2}$
and say that the partition {\it induces} the perfect matching.

\begin{lemma}\label{ipm} Let $A_1$ and $A_2$ be disjoint blocks in a partial Steiner triple system
$\mathcal{T}$ and let $V'=V\setminus(A_1\cup A_2)$.  The following three properties hold in $\mathcal{T}$.

{\rm (1).} If $\{\alpha,\beta,\gamma,x\}$ is a set of cardinality 4 in $V'$ containing no block with
$x$ as a member, then at most two sets of cardinality 9 containing $x,A_1\mbox{ and }A_2$
plus two members from $\{\alpha,\beta,\gamma\}$ have a partition into blocks.

{\rm (2).} If $\{\alpha,\beta,\gamma,x,y\}$ is a set of cardinality 5 in $V'$ containing no block
with both $x$ and $y$ as members, then to within labelling there is a unique way that the three
sets of cardinality 9 containing $x,y,A_1, A_2$ plus one element from $\{\alpha,\beta,\gamma\}$
have a partition into blocks.

{\rm (3).} Let $A$ be a 9-set with a partition into three blocks $A_1,A_2\mbox{ and }A_3$
and let $A'$ be any 9-set formed by replacing a point of $A_3$ with a point not in $A$.  If
either $A_1$ or $A_2$ must appear in a partition of $A'$ into three blocks, then $A'$ has no
partition into three blocks. 
\end{lemma}
\begin{proof} Consider (1) first.  There are three sets of the form described in the statement and
note that none of them have a partition with either $A_1$ or $A_2$ as a part because Lemma
\ref{six} would force a block containing $x$ and two elements from $\{\alpha,\beta,\gamma\}$.
Assume each of the three sets has a partition into three blocks which then induces three perfect
matchings in $K_{A_1,A_2}$.  We leave it up to the reader to verify that the three perfect
matchings are mutually edge-disjoint.

Label an edge of  $K_{A_1,A_2}$ with the other point in the block containing the edge.  We may
label the points of $A_1$ with $u_1,u_2,u_3$ and the points of $A_2$ with $v_1,v_2,v_3$ so
that the three edges with label $x$ are $u_1v_1,u_2v_2\mbox{ and }u_3v_3$.  The remaining
edges form a 6-cycle.  

In the partition with the block $[x,u_3,v_3]$, the edges for the other two parts must be
$u_1v_2$ and $u_2v_1$ because the three edges form a 3-matching.  Without loss of generality
we may assume $u_1v_2$ is labelled $\alpha$ and $u_2v_1$ is labelled $\beta$.  The partition
with the block $[x,u_2,v_2]$ forces the other two edges to be $u_1v_3$ and $u_3v_1$ with
one of them labelled $\gamma$.  Suppose it is $u_1v_3$.  Then the edge $u_3v_1$ must be
labelled $\alpha$.  This means that the partition containing the block $[x,u_1,v_1]$ must
involve the points $\beta$ and $\gamma$ but the edge $u_2v_3$ cannot receive either of
these two labels.  This contradiction proves (1).

For part (2), it is easy to see that there is no partition of any of the sets of the hypothesised form into
blocks with either $A_1$ or $A_2$ as parts.  So if all three have partitions, then there are
three edge-disjoint perfect matchings arising again but this time there is a perfect matching
with edges labelled $x$ and a perfect matching whose edges are labelled $y$.  The union
of these two perfect matchings is a 6-cycle and we may label points so that the 6-cycle is
$u_1v_1u_2v_2u_3v_3$ with $u_1v_1$ labelled $x$ (which determines the other labels).
The remaining edges are $u_1v_2$ labelled $\alpha$, $u_2v_3$ labelled $\beta$ and $u_3v_1$
labelled $\gamma$.  The actual partition is determined uniquely from this labelling.
This proves (2)

If $A_1$ must be in a partition of $A'$, then $A_2$ must be a block as well by Lemma
\ref{six}.  This would force the final block to contain two points of $A_3$ which is not
possible.  Part (3) then follows.    \end{proof}

\begin{lemma}\label{eleven} If $\mathcal{T}$ is a partial Steiner triple system of order 11 having
three disjoint blocks, then $\mathcal{T}$ is sequenceable.
\end{lemma}
\begin{proof} Let $|V|=11$ and arbitrarily label the points of $L$ with $a$ and $b$.  There are
at least six good points for the residual partial triple system on $V\setminus\{a\}$ and at least
six good points points for the residual partial triple system on $V\setminus\{b\}$. Thus, there is
a point that is good for both residual triple systems because the points are coming from a set of
cardinality 9.  Label the block containing such a point $B_3$ and label the point with 9.  Note
that this implies that $\{9,b\}$ is a good set for $V$.

We know there is another block containing two points that are good points for
$V\setminus\{b\}$ because there are at least six altogether.  Label such a block $B_1$ and
label two good points arbitrarily with 1 and 2 (thereby labelling 3).  Arbitrarily label the
points of $B_2$ with 4, 5 and 6, and the remaining two points of $B_3$ with 7 and 8.

Consider the sequence $\pi=b,1,2,4,3,5,7,6,8,a,9$.  Because 1 is a good point for $V\setminus\{b\}$,
the 9-segment $2,4,3,5,7,6,8,a,9$ cannot be partitioned into three blocks.  Because 9 is a
good point for both $V\setminus\{a\}$ and $V\setminus\{b\}$, neither of the other two
9-segments can be partitioned into three blocks.

The only possible remaining inadmissible segments are the two 3-segments $3,5,7$ and
$6,8,a$, and the 6-segment $3,5,7,6,8,a$.  First, let the 6-segment be admissible.  If both
$3,5,7$ and $6,8,a$ are admissible, then $\pi$ is admissible.  So assume that $[3,5,7]$ is
a block.  If $[5,8,a]$ is not a block, then interchanging the labels of 5 and 6 makes $\pi$
admissible.  If $[5,8,a]$ is a block, then interchange the labels of 7 and 8 and we again have
the resulting $\pi$ admissible.  We do a similar thing if we start with $[6,8,a]$ being a block.  

Second, let $3,5,7,6,8,a$ have a partition into two blocks.  Then interchange the labels of
4 and 5 giving us an admissible segment $3,5,7,6,8,a$ by Lemma \ref{six}.  We are back at
the preceding situation and this completes the proof.     \end{proof} 

\medskip

The next lemma sets the stage for the completion of the verification for three vertex-disjoint
blocks.  Of course, it does force us to concentrate on partial Steiner triple system of order 12. 

\begin{lemma}\label{stage} Let $\mathcal{T}$ be a partial Steiner triple system of order 12 with
three vertex-disjoint blocks but not four vertex-disjoint blocks, and let $$[1,2,3],[4,5,6],
[7,8,9],\{a,b,c\}$$ be a partition of $V$ into three blocks and three points.  If there is an
admissible sequence of the form $1,2,4,3,5,7,6,8,a,9,b,c$, then all partial Steiner triple systems
of order at least 13 having at most three vertex-disjoint blocks are sequenceble.
\end{lemma}
\begin{proof} Because of Lemmas \ref{zero} and \ref{two}, only partial Steiner triple systems with three
vertex-disjoint blocks, but not four, need be considered.  Consider a partial Steiner triple system
$\mathcal{T}'$ of order at least 13.  Choose a residual partial Steiner triple system consisting of the
three blocks $[1,2,3],[4,5,6],[7,8,9]$ and three other arbitrary points $a,b,c$.  Let $1,2,4,3,5,7,
6,8,a,9,b,c$ be an admissible sequence for the residual partial Steiner triple system.

Form a new sequence $\pi$ by adding an arbitrary permutation of the remaining points
$V\setminus\{1,2,3,4,5,6,7,8,9,a,b,c\}$ to the end of the given sequence.  We claim that
$\pi$ is an admissible sequence for the partial Steiner triple system $\mathcal{T}'$.

The sequence $\pi$ has no inadmissible 3-segments because such a segment would have to
start with $b$ or further along the sequence $\pi$.  However, this produces four disjoint
blocks in $\mathcal{T}'$ because none of $\{1,2,3,\ldots,9\}$ have been used.  Any
inadmissible 6-segment would have to begin with 8 or further along the sequence.  This
again would produce four disjoint blocks.  A similar argument works for 9-segments
because none of $\{1,2,3\}$ would be involved.  Therefore, $\pi$ is admissible for $\mathcal{T}'$
and the result follows.     \end{proof}

\section{Order Twelve}

This section deals with partial Steiner triple systems of order 12 having three vertex-disjoint blocks
but not four. This is a separate section because the proof that they are sequenceable is lengthy
and intricate.   We continue to use the notation that there are three disjoint blocks which will at
some point be labelled $B_1,B_2\mbox{ and }B_3$ and a set $L$ of three additional points.
The points of $B_1,B_2\mbox{ and }B_3$ will be labelled with $\{1,2,3\},\{4,5,6\}
\mbox{ and }\{7,8,9\}$, respectively, and the points of $L$ will be labelled $a,b,c$.

Throughout this proof we work to prove that there is a labelling of the points so that
the sequence $$\pi=1,2,4,3,5,7,6,8,a,9,b,c$$ is admissible.  This is done by providing an
initial labelling of the blocks as $B_1,B_2\mbox{ and }B_3$ and this labelling does not change.
The elements within the blocks will be given an initial labelling following the guidelines of the
preceding paragraph and various relabellings may take place as we work through the proof. 

We now provide an outline of the proof comprising this section.  The given sequence $\pi$
contains four 3-segments, three 6-segments and four 9-segments which may be inadmissible.
They are:
\begin{itemize}
\item 3-segments: $3,5,7$;\:\:$6,8,a$;\;\;$a,9,b;$\;\;and $9,a,b$;
\item 6-segments: $3,5,7,6,8,a$;\;\;$7,6,8,a,9,b$;\;\;and $6,8,a,9,b,c$; and
\item 9-segments:  all four of them.
\end{itemize}  
The first result below gives a flexible choice for a particular 6-segment and
this 6-segment is of interest because it contains two of the possible inadmissible 3-segments.
Thus, if the 6-segment is admissible, all the 3-segments contained in it are admissible.

We then move to considering 9-segments.  Of course, the 9-segments are the most complicated
because they may have more than one partition into blocks, whereas, an inadmissible 6-segment
has a unique partition into blocks and a 3-segment either is or is not a block itself.  We give two
results that provide a way of guaranteeing that none of the four 9-segments have partitions
into three disjoint blocks and leave only a couple of potential problematical blocks.  To eliminate
the problematical blocks, we consider a variety of cases.

\begin{lemma}\label{9flex} Arbitrarily label the blocks $B_1,B_2\mbox{ and }B_3$.  If we
label an arbitrary point of $L$ with $a$ and an arbitrary point of $B_1$ with 3, then we
may label the points of $B_2$ and $B_3$  so that the segment $3,5,7,6,8,a$ is admissible,
where there are at least two choices for the point to be labelled 9.
\end{lemma}
\begin{proof} First label arbitrary points of $L$ and $B_1$ with $a$ and 3, respectively.
By Lemma \ref{six}, at most one of the sets $\{a,3,4,5,6,x,y\}$, over the three choices
of $\{x,y\}\subset B_3$, has a partition into two blocks.  That gives at least two choices
for the point labelled 9 so that $\{a,3,4,5,6,7,8\}$ has no partition into two blocks.  

Consider the sequence $\pi'=3,5,7,6,8,a$ for a choice that has no partition of the 6-set. The
only possible inadmissible 3-segments contained in $\pi'$ are $3,5,7$ and $6,8,a$.  If both
are admissible, then $\pi'$ is admissible and we are done.  If $[3,5,7]$ is a block and $[5,8,a]$
is not a block, then interchanging the labels of 5 and 6 results in $\pi'$ being properly admissible.
If both $[3,5,7]$ and $[5,8,a]$ are blocks, then interchanging the labels of 5 and 6,
and 7 and 8 results in $\pi'$ being admissible (noting that $[3,6,7]$ is not a block prior to
relabelling because $[5,8,a]$ is a block and $3,5,7,6,8,a$ has no partition into two blocks).
A similar argument works when $[6,8,a]$ is a block.  This completes the proof.    \end{proof}

\bigskip

Recall that a set $A$ of cardinality 3 is a good set if $V\setminus A$ has no partition into
three blocks.  We are interested in the complements of the four 9-segments, that is,
$\{1,2,4\},\{1,2,c\},\{b,c,1\}\mbox{ and }\{b,c,9\}$ as we require all four of them to be
good.  If there is an $\alpha\in B_i$ such that $\{b,c,\alpha\}$ is a good set, we say that
$\alpha$ is an $a$-{\it replacement}.  It is then clear what $b$- and $c$-replacements are.

Our next goal is to show there is a choice of $a$ for which $B_1,B_2\mbox{ and }B_3$ all
have at least two $a$-replacements.  The next result achieves the goal.

\begin{lemma}\label{rplc} There is a point in $L$ which may be labelled $a$ such that each
$B_i$, $1\leq i\leq 3$, has at least two $a$-replacements. 
\end{lemma}
\begin{proof} If the result is not true, then for each $x\in L$ there is a $B_i$ such
that two of the numbers in $B_i$ are not $x$-replacements.  Without loss of generality we
label the points of $B_1$ with 1, 2, 3  so that neither 2 nor 3 are $b$-replacements. This means
that both $$\{1,2,4,5,6,7,8,9,b\}\mbox{ and }\{1,3,4,5,6,7,8,9,b\}$$ have partitions into
three blocks.  It is easy to see that neither $B_2$ nor $B_3$ can be parts in either partition.
So the two partitions generate perfect matchings in $K_{B_2,B_3}$ and it is easy to check
that the two perfect matchings are edge-disjoint.

We may assume the partition of $\{1,2,4,5,6,7,8,9,b\}$ is $[1,5,8],[2,6,9]$ and $[4,7,b]$
without loss of generality.  A partition of $\{1,3,4,5,6,7,8,9,b\}$ is determined once the block
containing $b$ is known.  If the block is $[5,9,b]$, then the other two blocks are forced to be
$[1,6,7]\mbox{ and }[3,4,8]$, whereas, if the block is $[6,8,b]$, then the remaining two blocks
are $[1,4,9]\mbox{ and }[2,6,7]$.  However, if we relabel the vertices for the latter situation
using the permutation $(2\;\;3)(6\;\;4\;\;5)(8\;\;7\;\;9)$, we obtain the same set of blocks as
the first situation.  Thus, we may assume that 
\[[4,7,b],[5,9,b],[1,5,8],[1,6,7],[2,6,9]\mbox{ and }[3,4,8]\] are blocks realizing the edges
$47,48,58,59,67\mbox{ and }69$ in the union of the two perfect matchings.

We now wish to show that if 3 also is not an $x$-replacement, where $x\in\{a,c\}$, then 1 and
2 are $x$-replacements.  So let 3 also not be an $x$-replacement.  This means that
$\{1,2,4,5,6,7,8,9,x\}$ has a partition into blocks.  Look at the block containing 1.  It cannot be
$[1,5,8]$ as this would force $[4,7,x]$ to be a block which is impossible.  It cannot be $[1,6,7]$
because this prevents 5 from belonging to a block.  Hence, the block must be $[1,4,9]$ which
admits $[2,5,7]\mbox{ and }[6,8,x]$ as the only completion to a partition.

The three partitions use all of the edges of $K_{B_2,B_3}$ in blocks.  We now need to show
that both 1 and 2 are $x$-replacements. The set $\{2,3,4,5,6,7,8,9,x\}$ has no partition
because the only two blocks involving 2 are $[2,5,7]$ and $[2,6,9]$.  The first one forces
$[3,4,9]$ to be a block and the second leaves no block to contain $x$.  Hence, 1 is an
$x$-replacement.  Then $\{1,3,4,5,6,7,8,9,x\}$ also has no partition because the only block
containing 3 is $[3,4,8]$ and the only block containing $x$ is $[6,8,x]$.  Therefore, both 1
and 2 are $x$-replacements.  

The preceding argument slightly modified also works with 2 and 3 interchanging roles.
We now are able to conclude that $B_1$ contains at least two $x$-replacements for all
$x\in\{a,c\}$.  If both 2 and 3 are $x$-replacements, the conclusion follows, whereas,
if one of 2 or 3 is not an $x$-replacement, then the remaining two numbers in $B_1$ are
$x$-replacements validating the conclusion. Therefore, the only way the conclusion of the
lemma fails is if there are two numbers in $B_2$ that fail to be $x$-replacements
for $x\in\{a,c\}$ and two numbers in $B_3$ that fail to be $y$-replacements, where
$\{x,y\}=\{a,c\}$.  We now show this cannot happen.

There are cases to consider and we work through one case to illustrate how the proof
works.   All we know is the existence of the six blocks earlier in the proof. We examine
$B_2$ starting with 5 with $x\in\{a,c\}$.  Look at the edges incident with 6 for possible
partitions of $\{1,2,3,4,6,7,8,9,x\}$. The edge 67 belongs to the block $[1,6,7]$ and
its presence forces the block $[3,4,8]$ to be there as well.  Then $[2,9,x]$ would have
to be a block but the block containing 29 is $[2,6,9]$. 

In a similar fashion one can show that there is no partition involving the edges 68 and
69.  This implies the set $\{1,2,3,4,6,7,8,9,x\}$ has no partition into blocks which means
that 5 is an $x$-replacement.

Examining the set $\{1,2,3,5,6,7,8,9,x\}$ leads to the only possible partition being
$[1,5,8],[2,6,9]\mbox{ and }[3,7,x]$.  There are several possible partitions of
$\{1,2,3,4,5,7,8,9,x\}$.  We must take each of them and pair it with the preceding
partition of $\{1,2,3,5,6,7,8,9,x\}$ and show that the two of them imply that at least
two numbers from $B_3$ are $y$-replacements as this would give two or more numbers
from each $B_i$, $i=1,2,3$, as $y$-replacements.

One possible partition of $\{1,2,3,4,5,7,8,9,x\}$ is $[1,4,9],[2,8,x]$ and $[3,5,7]$.
Combining this partition with that for $\{1,2,3,5,6,7,8,9,x\}$ we find that both 7 and
8 are $y$-replacements as required.  The other possible partition for
$\{1,2,3,4,5,7,8,9,x\}$ is $[1,9,x],[2,5,7],[3,4,8].$  It turns out that 7 and 9 are
$y$-replacements and that completes the proof.    \end{proof}

\bigskip

Let's take a moment to see the current status of our progress.  As there are two choices
for the element of $B_3$ to be labelled 9 for Lemmas \ref{9flex} and \ref{rplc}, there is at least
one element of $B_3$ that can be labelled 9 satisfying both lemmas.  This makes many, but not
all, of the problematic segments admissible.

\begin{lemma}\label{crucial} Suppose there is a labelling for which $\{1,2,c\}$ and $\{1,2,4\}$
are good sets, and both Lemmas \ref{9flex} and \ref{rplc} hold.  If none of the following
3-sets are blocks, then $\pi$ is admissible: $\{6,a,b\}, \{9,a,b\}, \{9,b,c\}$ and $\{6,8,b\}$
when $\{9,a,c\}$ is a block.
\end{lemma}
\begin{proof} Because both $\{1,2,4\}$ and $\{1,2,c\}$ are good sets, the 9-segments
ending with $b$ and $c$ do not have partitions into three blocks.  Because of Lemma
\ref{rplc}, either 1 or 2 is an $a$-replacement so we may choose to label them so that 1
is an $a$-replacement and the 9-segment beginning with 2 has no partition into three
disjoint blocks. 

Similarly, we have two choices for the point labelled 9 for Lemmas \ref{9flex} and \ref{rplc}.
This gives us that the remaining 9-segment has no partition into three disjoint blocks
and that the segment $3,5,7,6,8,a$ is admissible.  The latter implies that $\{3,5,7\}$ and
$\{6,8,a\}$ are not blocks.

Thus, the segments that may be inadmissible are $a,9,b$; $9,b,c$; $7,6,8,a,9,b$; and
$6,8,a,9,b,c$.  The first two are  admissible by hypothesis.  The segment $7,6,8,a,9,b$
is inadmissible only if $[6,8,b]$ is a block and it is not.  The only possible partition of the
6-segment $6,8,a,9,b,c$ is into the blocks $[6,8,b]$ and $[9,a,c]$ which is not the case
by hypothesis.    \end{proof}

\bigskip  

We now have set the stage for the proof of the main theorem.  The primary idea is to
obtain a labelling in accordance with Lemma \ref{crucial} and do so in a way that avoids the
blocks listed in the lemma.  There are cases depending on the distribution of blocks composed
of two points from $L$ and one point from a $B_i$.  Such a block is called a {\it tiple} and
the point from the $B_i$ is called the {\it tip}.  

\begin{thm}\label{main} A partial Steiner triple system $\mathcal{T} $with at most three
disjoint blocks is sequenceable.
\end{thm}
\begin{proof} Lemmas \ref{zero} and \ref{two} take care of the cases that $\mathcal{T}$
has at most two disjoint blocks.  If there are three disjoint blocks, but not four, then
Lemmas \ref{nine}, \ref{ten} and \ref{eleven} take care of the cases that $\mathcal{T}$
has order 9, 10 or 11.  Finally, Lemma \ref{stage} tells us that if we can find an admissible
sequence of the form $\pi=1,2,4,3,5,7,6,8,a,9,b,c$ when $\mathcal{T}$ has order 12,
then the result holds for all larger orders.  The rest of the proof consists of arguments
verifying there is a labelling so that $\pi$ is admissible.

{\bf Case 1}. There are three tiples all of whose tips are in the same block $B_i$.
The first labelling step is giving label $B_2$ to the block containing the tips of
the three tiples, and arbitrarily labelling the other two blocks $B_1$ and $B_3$.  Label
the point of $L$ satisfying Lemma \ref{rplc} with $a$, and arbitrarily label the other two
points $b$ and $c$.  By the pigeonhole principle, there is a point in $B_3$ we label 9 
so that both Lemma \ref{9flex} applies and it is an $a$-replacement.  Label the point
in $B_2$ with 4 so that $[4,a,b]$ is a block.

{\bf Claim 1}. At most one of the sets $\{1,2,c\},\{1,3,c\}\mbox{ and }\{2,3,c\}$ is bad
for any labelling of the points of $B_1$.  Suppose that $\{\alpha,4,5,6,7,8,9,a,b\}$, $\alpha\in\{1,2,3\}$,
has a partition into three blocks.  Neither $B_2$ nor $B_3$ are parts in the partition as this
would imply that $[\alpha,a,b]$ is a block.  The partition generates a perfect matching in
$K_{B_2,B_3}$ with one edge labelled $\alpha$, one labelled $a$ and the other labelled $b$.

If there is a partition of $\{\beta,4,5,6,7,8,9,a,b\}$, $\beta\neq\alpha$, then the
perfect matching generated is edge-disjoint from the first perfect matching by part (3)
of Lemma \ref{ipm}.  It has an edge labelled $\beta$ incident with 4 leaving only one choice
for the edge labelled $a$.  This leaves no choices for an edge to be labelled $b$ and
the claim follows.

{\bf Claim 2}.  At most one of the sets $\{1,2,4\},\{1,3,4\}\mbox{ and }\{2,3,4\}$ is
bad under any labelling of the points of $B_1$.  Supposing that $\{\alpha,5,6,7,8,9,a,b,c\}$,
$\alpha\in\{1,2,3\}$, has a partition into blocks, it is easy to verify that none of the parts
may be a tiple or $B_3$.  We may label the remaining two points of $B_2$ with 5 and 6 so
that $[5,a,c]$ and $[6,b,c]$ are the other two tiples.  Then $B_3$ and $B'= [5,a,c]$ are
disjoint blocks so that a partition generates a perfect matching in $K_{B_3,B'}$.  This forces
one of the blocks to contain $\{6,a\}$.  But this is the same for any $\alpha\in\{1,2,3\}$
and part (3) of Lemma \ref{ipm} then implies none of the other two possibilities have a
partition into blocks.

From the two claims we conclude that there is a labelling of $B_1$ so that both $\{1,2,c\}$
and $\{1,2,4\}$ are good sets.  There are no tiples with tips in $B_3$ so that the only
possible problem 3-set in Lemma \ref{crucial} is $\{6,a,b\}$.  However, the relabelling
done in Lemma \ref{9flex} does not affect 4 so that $[4,a,b]$ still is a block which means
$\{6,a,b\}$ is not a block.  Thus, $\pi$ is admissible for this case.

\medskip

{\bf Case 2}. There are two tiples, but not three, whose tips are in the same $B_i$.
Label the block containing two tips $B_2$.  Label the point of $L$ satisfying Lemma \ref{rplc}
with $a$, and do not label the other two points of $L$ at this time.  There may or may not
be a third tiple.  It there is, label the block containing its tip $B_1$.

If the point labelled $a$ is in both blocks with tips in $B_2$, then label points so
that $[5,a,b]$ and $[6,a,c]$ are blocks which, of course, means that when the third
tiple exists it contains $b$ and $c$.  Claim 1 of the preceding proof still holds
so that two of the sets $\{1,2,c\},\{1,3,c\}\mbox{ and }\{2,3,c\}$ are good.
Claim 2 also remains valid except that we are using $[5,a,b]$ and $[6,a,c]$ in place
of $[5,a,c]$ and $[6,b,c]$, respectively. 

Note that $[6,8,a]$ cannot be a block because $[6,a,c]$ is a block.  Thus, when we
apply the relabelling of Lemma \ref{9flex} to make the segment $3,5,7,6,8,a$
admissible, the labels of 5 and 6 do not change so that $[5,a,b]$ remains a block
so that $\{6,a,b\}$ is not.  Lemma \ref{crucial} implies that $\pi$ is admissible.  

The preceding case was for $a$ belonging to two tiples whose tips are in the same
block which is again labelled $B_2$.  Now let $a$ belong to just one of these two
tiples.  This subcase is the most intricate portion of the proof.  Label the other points of
$L$ so that $b$ is the point belonging to both tiples with tips in $B_2$.  Label the points
of $B_2$ so that $[4,a,b]$ and $[6,b,c]$ are blocks.  If there is a third tiple label the block
containing its tip $B_1$ and arbitrarily label the points of $B_1$.  Claim 1 still is valid
via essentially the same argument.

If at least two of the sets $\{1,2,4\},\{1,3,4\}\mbox{ and }\{2,3,4\}$,  are
good, then we complete the argument as  before.  If this is not the case we need a
variation.  If all three of the sets $\{1,2,5\},\{1,3,5\}\mbox{ and }\{2,3,5\}$ are
good, we may assume that $\{1,2,5\},\{1,2,c\},\{1,3,5\}\mbox{ and }\{1,3,c\}$ are
good as Claim 1 still holds. First consider $\{1,2,5\}$ and $\{1,2,c\}$ and choose 1 as the
$a$-replacement by relabelling if necessary.  It is straightforward to verify that the sequence
$\pi'=1,2,5,3,4,7,6,8,a,9,b,c$ is admissible when the segment $3,4,7,6,8,a$ is admissible.

When the segment $3,4,7,6,8,a$ is inadmissible, the relabelling we are allowed to carry
out is interchanging 7 and 8.  This then makes $\pi'$ admissible except in the following
situations: 1) Both $[3,4,7]$ and $[6,8,a]$ are blocks; 2) both $[3,4,8]$ and $[6,7,a]$ are
blocks; 3) both $[3,4,7]$ and $[6,7,a]$ are blocks; and 4) both $[3,4,8]$ and $[6,8,a]$
are blocks.

We then turn to the good sets $\{1,3,5\}$ and $\{1,3,c\}$.  Suppose that 1 and 2 did not
switch labels in the preceding step to obtain 1 as an $a$-replacement.  Then 1 still is an
$a$-replacement here.  Now interchange the labels of 2 and 3.  The segment
$3,4,7,6,8,a$ now is admissible because either $[6,7,a]$ or $[6,8,a]$ still is a block in
all four of the problematic block scenarios of the preceding paragraph, but $[2,4,7]$
and $[2,4,8]$ are not blocks.  Thus, $\pi'$ is now admissible.

If 1 and 2 did switch labels in the first step because 1 is not an $a$-replacement, then
$\{1,3,c\}$ and $\{1,3,5\}$ become $\{2,3,c\}$ and $\{2,3,5\}$, respectively.  Then
3 is the $a$-replacement and the rest of the argument carries through in the obvious way.

To complete this troublesome subcase, we start at the point that $[4,a,b]$ and $[6,b,c]$
are blocks and rebuild other assumptions.  It is easy to show that at least one of the
sets $\{1,2,4\},\{1,3,4\}\mbox{ and }\{2,3,4\}$ is good and from above, if two of them
are good, we are done.  Hence, we may assume that both $\{1,2,4\}$ and $\{1,3,4\}$
are bad.  One of them must involve $[7,8,9]$ as a block so that we may assume the
partition consists of the blocks $[7,8,9]$, $[6,b,c]$ and $[2,5,a]$.  The other partition
is forced to have $a,6$ in the same block and we assume that the blocks are $[6,7,a],
[5,8,b]\mbox{ and }[3,9,c]$.  Note that this labels $B_3$.

Now consider the three sets $\{1,2,5\},\{1,3,5\}\mbox{ and }\{2,3,5\}$.  If the 9-set
$\{\beta,4,6,7,8,9,a,b,c\}$, $\beta\in\{1,2,3\}$, has a partition into blocks, then $[7,8,9]$
cannot be one of the parts as both $[4,a,b]$ and $[6,b,c]$ are contained in the remaining
six points.  Thus, each part in a partition must be a block intersecting one point of $\{7,8,9\}$,
one point of $[4,a,b]$ and one point of $[6,b,c]$.  In particular, this forces $a$ to be in
a block with 6, that is, the block $[6,7,a]$.  The other blocks must be $[4,8,c]$ and $[\beta,9,b]$.
 
This implies there is a partition for at most one $\beta\in\{1,2,3\}$.  The case for no
partitions already has been done so we may assume exactly two of the sets $\{1,2,5\},
\{1,3,5\}\mbox{ and }\{2,3,5\}$ are good.

If $[\beta,4,7]$ is not a block, then $\beta,4,7,6,8,a$ is admissible. Furthermore, if
$\beta,4,7,6,9,a$ is inadmissible, then $[\beta,4,9]$ is a block and $\beta,4,8,6,9,a$ is 
admissible. Thus, after relabeling $1,2,4,3,5,7,6,8,a,9,b,c$ is an admissible sequence. Hence,
we will assume that $[\beta,4,7]$ does form a block, but we will use the label $\gamma$ 
instead of $\beta$ to keep using $\beta$ as an unlabelled element in our arguments.

Assume that $[\gamma,4,7]$ is a block for some $\gamma\in\{1,2,3\}$. Look at the sets
$\{1,2,6\},\{1,3,6\}$ and $\{2,3,6\}$.  Assume $\{\beta,4,5,7,8,9,a,b,c\}$ can be partitioned
into three blocks for some $\beta\in\{1,2,3\}$.  As $[4,a,b]$ and $[4,8,c]$ are blocks,
no block of the partition may contain $\{\beta,4\}$ as it would leave the three letters for the
remaining two blocks, but there are no tips in $\{5,7,8,9\}$. Because $[2,5,a]$ and $[5,b,8]$
are blocks, $[4,8,c]$ cannot be a block of the partition as this would force either a block
containing the pair $\{7,9\}$, or one of 7 and 9 to be a tip. Therefore, any partition of
$\{\beta,4,5,7,8,9,a,b,c\}$ must use $[4,a,b]$ which forces $[7,8,9]$ and $[\beta,5,c]$
to be the remaining parts.  We conclude that at least two of $\{1,2,6\},\{1,3,6\}$ and
$\{2,3,6\}$ are good sets.

Consider the segment $\alpha,4,8,5,\delta,a$ for any $\alpha\in\{1,2,3\}$ and $\delta\in\{7,9\}$.
The set $\{\alpha,4,8\}$ is not a block as $[4,8,c]$ is a block. Similarly, $\{5,\delta,a\}$ is not
a block because $[2,5,a]$ is a block. Finally, neither $\{\alpha,5,8\}$ nor $\{5,8,a\}$
form a block because $[5,8,b]$ is a block. Therefore, the segment is admissible. This yields two
choices for 9 so that we may relabel and obtain an admissible sequence of the correct form.

\medskip

{\bf Case 3}.  There is at least one tiple and no two with tips in the same $B_i$.
We are assuming there are one, two or three tiples. First label a point of
$L$ with $a$ to accommodate Lemmas \ref{9flex} and \ref{rplc}.  If $a$ belongs to no tiple,
then there is only one tiple and we arbitrarily label the remaining points in $L$ with $b$ and $c$,
and the points in $B_2$ so that $[5,b,c]$ is the tiple.  Arbitrarily label the remaining two blocks
$B_1$ and $B_3$.

We look at 9-sets of the form $\{\alpha,4,5,6,7,8,9,a,b\}$, where $\alpha\in B_1$, and the labels
for $B_3$ are not yet chosen.  As before, when a set of this form has a partition into three blocks,
it must generate a perfect matching in $K_{B_2,B_3}$.  The pair $5,b$ appears in the block $[5,b,c]$
so that no perfect matching has an edge labelled $b$ incident with the point 5.  This means there
are at most two perfect matchings so that there is at least one set of the form under discussion
which does not have a partition into three blocks.

Now consider the situation that the following two 9-sets $$\{\alpha,4,5,6,7,8,9,a,b\}
\mbox{ and }\{\beta,4,5,6,7,8,9,a,b\},$$ where $\alpha,\beta\in B_1$ and $\alpha\neq\beta,$
have partitions into three blocks. Without loss of generality, we may assume $B_3$ is labelled so
that the first perfect matching has the edge 47 labelled $\alpha$, the edge 58 labelled $a$ and the
edge 69 labelled $b$.  In the second perfect matching, the label $b$ must be on the edge 48 which
forces the edge 67 to be labelled $a$, thereby, leaving the edge 59 to be labelled $\beta$.

Given the edges as labelled in the preceding paragraph, we then have that $\{\alpha,\beta,c\}$
is a good set.  If we show that $\{\gamma,5,6,7,8,9,a,b,c\}$ has no partition into blocks, where
$B_1=\{\alpha,\beta,\gamma\}$, then $\{\alpha,\beta,4\}$ also is a good set.  The block $B_3$
cannot be a part in a partition as this would force $[5,b,c]$ to be a part which, in turn, would
imply that $[\gamma,6,a]$ is a block.  This is not possible because $[6,7,a]$ is a block.  

Thus, any partition of the 9-set must generate a perfect matching in $K_{B_3,B}$, where $B=
[5,b,c]$.  The points $a$ and 5 must be in the same block because neither $a,b$ nor $a,c$ are
in a block.  So the block is $[5,8,a]$.  Lemma \ref{six} forces $[6,9,b]$ to be a block leaving
$\gamma,4,7$.  However, $[\alpha,4,7]$ is a block and the 9-set under discussion has no
partition into three blocks.

We now label $B_1$ so that both $\{1,2,c\}$ and $\{1,2,4\}$ are good sets and 1 is an
$a$-replacement.  Now that 3 has been labelled we label 9 so that it is an $a$-replacement.
Lemma \ref{crucial} now implies that $\pi$ is admissible because none of the problematic
3-sets are blocks.

\medskip

If there is only one tiple and it contains $a$, then we label the other point of the tiple in $L$
with $c$ and choose $B_2$ as we just did with the tiple being $[5,a,c]$.  The proof now goes
through the same with $a$ playing the role of $b$ in the preceding proof.

\medskip

If there are two tiples and $a$ belongs to just one of them, then do the obvious labelling
so that $[5,b,c]$ is one tiple and $[\alpha,a,c]$ is the other with $\alpha\in B_1$.  The 
preceding argument still works because the tiple $[\alpha,a,c]$ restricts the perfect
matchings even more 

\medskip
 
We must be more careful when $a$ belongs to both tiples because there will be a tiple
containing both $a$ and $b$.  What we do in this case after labelling $a$ is to label the
given blocks so that one tip is in $B_2$ and the other is in $B_3$.  Label the remaining
points of $L$ and a point of $B_2$ so that $[4,a,b]$ is a block which implies that $a$
and $c$ are the points of the tiple whose tip lies in $B_3$. 

It takes a little work to show there is a labelling such that both $\{1,2,c\}$ and
$\{1,2,4\}$ are good sets.  If all three of the sets of the form $\{\alpha,\beta,c\}$,
$\alpha,\beta\in B_1$, are good, then use the fact that at least one set of the form
$\{\alpha,\beta,4\}$ is good to establish a labelling of $B_1$ that works.

If not all three are good, it is straightforward to show that at least two of the sets of
the form $\{\alpha,\beta,c\}$ are good.  Then assuming that the 9-set $\{\alpha,4,5,6,7,8,9,a,b\}$
has a partition into three disjoint blocks, we can show there are at least two good sets
of the form $\{\alpha,\beta,4\}$, $\alpha,\beta\in B_1$.  The key is to show there is
no partition of a 9-set of the form $\{\alpha,5,6,7,8,9,a,b,c\}$ with $[x,a,c]$ as a
part for $x\in B_3$.  Then work with the perfect matchings arising in
$K_{B_3,L}$ for possible partitions.

Hence, we label the points of $B_1$ so that $\{1,2,c\}$ and $\{1,2,4\}$ are good sets,
and choose 9 so that Lemma \ref{9flex} is valid.  The sequence $\pi$ is now admissible
from Lemma \ref{crucial} because the tiple $[4,a,b]$ does not change in any relabelling.            

\medskip

When there are three tiples, we choose $a$ in the usual way and arbitrarily label the other
two elements of $L$ with $b$ and $c$.  We label $B_2$ so that it contains the tip in the
tiple containing $a,b$.  We label 4 so that $[4,a,b]$ is the tiple.  We then label $B_1$
for the tip of the tiple containing $b,c$  and $B_3$ for the tiple containing $a,c$.  This
puts more restrictions on the perfect matchings arising from partitions of certain 9-sets
so the argument for the preceding case still works.

This completes the proof of Case 3.

\medskip     

{\bf Case 4}.  There are no tiples. Because there are no tiples, Lemma \ref{crucial} means
we need only find a labelling so that $\{1,2,c\}$ and $\{1,2,4\}$ are good sets.  To do
this we work with the distribution of the perfect matching edges.  If there is a point
$x\in L$ such that there are three edges in $K_{B_i,B_j}$ labelled $x$, then there are
no edges labelled $x$ incident with a point of $B_k$, $L=\{i,j,k\}$.  So label $x$ with $a$
and make $B_2=B_k$ and arbitrarily label $B_1$ and $B_3$.  There are no edges 
labelled with $a$ between $B_2$ and $B_3$ so that all sets of the forms $\{1,2,c\}$ and
$\{1,2,4\}$ are good.  Lemma \ref{crucial} then implies that $\pi$ is admissible.

Because the set $B_i,B_j,B_k$ contain nine points, there are at most four edges with
any of the labels coming from points of $L$.  The preceding paragraph tells us we may
assume the distribution between the three pairs of those three blocks is at most 1, 1, 2.
Thus, there must be at least two points $x,y\in L$ such that for some pair $B_i,B_j$
both $x$ and $y$ have at most one edge labelled with $x$ and $y$, respectively,
between the two blocks.  Hence, if we label the two blocks $B_2,B_3$ and the points
1 and 2, then both $\{1,2,c\}$ and $\{1,2,4\}$ are good sets.  Lemma \ref{crucial}
then shows that $\pi$ is admissible.    \end{proof}

\section{Further Considerations}

We have seen that not all partial Steiner triple systems are sequenceable and there
may be a connection between how many disjoint blocks such a system possesses and
sequenceability.  The following encapsulates the thought in a question.

Define a function $T$ as follows.  For each positive integer $k$, let $T(k)$ be the smallest integer
so that if $\mathcal{T}$ is any partial Steiner triple system with $k$ disjoint blocks, but not
$k+1$, and at least $3k+T(k)$ vertices, then $\mathcal{T}$ is sequenceable.  Let's see that
this is a well-defined function.

Suppose that $\mathcal{T}$ has $k$ disjoint blocks, but not $k+1$ disjoint blocks, and has
order at least $15k-5$.  Call the points of $k$ disjoint blocks $U=\{u_1,u_2,\ldots,u_{3k}\}$ and
let the remaining points be $V=\{v_1,v_2,\ldots,v_{12k-5},\ldots\}$.  Define a sequence by starting
with $u_1,u_2,u_3,\ldots,u_{3k}$ and then between $u_i$ and $u_{i+1}$, $1\leq i\leq 3k-1$,
insert five arbitrary points from $V$. Any points left over from $V$ are tacked on at the end.
 Call the resulting sequence
$\sigma$.  It is clear that no segment of length $3m$, $m>1$, can have a partition into
$m$ blocks because there are at most $m-1$ points from $U$.
This means at least one of the blocks must contain only points from $V$ which is a contradiction.

Segments of length 3 may be blocks but they may be eliminated by simple interchanges.  For
example, if $[v_{5i},u_i,v_{5i+1}]$ is a block, then interchange $v_{5i-1}$ and $v_{5i}$ and
the new segment is not a block.  There are several types but all are easily modified to finally
obtain an admissible sequence.

\medskip

In this paper we have proven that $T(1)=T(2)=T(3)=0$ and $T(4)>0$.  An obvious problem is:
what can be said about the function $T$?

\end{document}